\documentclass[11pt, a4paper]{amsart}
\usepackage{amsmath}
\usepackage{mathdots}

\newtheorem{theorem}{Theorem}[section]     % (numérotation croissante, pour éviter Lemma 9.1 et Th. 9.1)

\newtheorem{lemma}[theorem]{Lemma}
\newtheorem{proposition}[theorem]{Proposition}
\newtheorem{corollary}[theorem]{Corollary}

\usepackage[colorlinks=true,citecolor=cyan,backref=page]{hyperref}

\usepackage[vcentermath]{youngtab}
\usepackage{ytableau}
\usepackage[all]{xy}

\usepackage{xspace}

\usepackage{shuffle}
\usepackage{ulem}

\usepackage{tikz}
\usetikzlibrary{patterns}
\usetikzlibrary{fadings}
\tikzfading[name=fade out, inner color=transparent!0, outer color=transparent!90]
\usepackage[labelfont=normalfont,labelformat=simple]{subcaption}
\usepackage{graphicx}

\usepackage[vcentermath]{youngtab}
\usepackage{ytableau}

\usepackage{enumerate}

\title{On a lemma of Schensted}

\author{A. Abram}
\address{Antoine Abram,
D\'epartement de math\'ematiques, Universit\'e du Qu\'ebec \`a Montr\'eal}
\email{abram.antoine@courrier.uqam.ca}

\author{C. Reutenauer}
\address{Christophe Reutenauer,
D\'epartement de math\'ematiques, Universit\'e du Qu\'ebec \`a Montr\'eal}
\email{Reutenauer.Christophe@uqam.ca}

\date{\today}

\begin{document}

%\keywords{plactic monoid, stylic monoid, tableaux, partitions, standard immaculate tableaux, evacuation, $J$-trivial, $J$-order}

\begin{abstract} We give a direct proof of Schensted's lemma asserting that row and column insertion in a tableau commute.
\end{abstract}
 
\maketitle

\tableofcontents

%\subjclass{20M05, 20M20, 20C30, 05E10, 06F05}
%\endsubjclass

\section{Introduction}

The purpose of this Note is to give a complete direct proof of Schensted's Lemma 6 in \cite{S}. This result asserts that row insertion and 
column insertion commute. This nice commutation result, besides its own interest, has many applications. One of them is that reversing the 
permutation amounts to transpose its $P$-tableau (Theorem 3.2 in \cite{Sa}, a result which is implicit in Schensted's article). Moreover, many 
proofs in the theory of the Schensted algorithm (also called Robinson-Schensted, or Robinson-Schensted-Knuth, RSK) use this commutation result.

Schensted gives a direct proof of his lemma; however many cases are not considered  and not even listed, so that the proof is incomplete. 
Similarly in \cite{Sa} (Proposition 3.2.2)\footnote{The present Note may be seen as an addendum to the very nice book of Bruce Sagan \cite{Sa}.}. Schensted's lemma has indirect proofs, see for example \cite{St} (Lemma 7.23.14 in Stanley's book, and Corollary 
A1.2.11 in the Appendix by Fomin), or \cite{vL} Theorem 4.1.1. It follows also from the results of Schützenberger relating Schensted's 
algorithm and jeu de taquin \cite{Sc}. Moreover, the theory of the plactic monoid also implies Schensted's lemma, since column and row insertions correspond to left and right products in this monoid  (see the chapter by Lascoux, Leclerc and Thibon on the plactic monoid in the book by Lothaire \cite{LLT}).

We sketch now our direct proof of Schensted's lemma. Note that we do not argue by induction using the maximum element of the tableau, as is done by Schensted and Sagan.

The ideas of the proof are as follows. 
We carefully define the {\it trail} of an insertion, which is a sequence of boxes with their labels; this is called elsewhere the ``insertion path". 
We consider the trail as embedded in the tableau 
{\it before} its modification by the insertion. Note that the modification of the tableau by the insertion is obtained by sliding the labels along the 
trail.

Now, we have two insertions: a column insertion, and a row insertion; hence two trails, a row trail and a column trail. We study how these two 
trails intersect. If they do not intersect, the commutation 
is easy to see. 

Suppose now that the two trails intersect: we show that they have exactly one box in common, and do not intersect elsewhere, see Figure 
\ref{croisement}. The tableau obtained by the two insertions, in either order, is obtained as follows.
First, observe that if one wants to slide the labels along the two trails (row and column), then at the neighbourhood of the intersection, one has 
conflicts between boxes. Precisely, the intersection box could get two labels, and there is only one label for the two boxes (in the two trails) after the intersection box.
We prove that a simple rule, depending on two cases, solves the conflict (see Figure \ref{Cases}, where $s$ is the label of the intersection, and 
where the labels of the column trail, resp. of the row trail, are $\ldots,a,s,b,\ldots$, resp. $\ldots,i,s,j,\ldots$), and that for the other boxes, one 
performs the ordinary sliding of the two trails.

\section{Schensted insertions}

Recall that a {\it tableau} is a finite lower order ideal of $\mathbb N\times \mathbb N$, where the latter is partially ordered by the usual componentwise order ($(a,b)
\leq(u,v)$ if and only if $a\leq u$ and $b\leq v$), together with an increasing injective mapping from this subset into a totally ordered set 
(usually the latter set is $\mathbb N$).

Recall that the {\it row insertion} of an element $x$ into a row $L$ (with $x\notin L$), is the row obtained by adding $x$ at the end of $L$ if 
$x>\max(L)$ or if the row is empty, and if this condition is not satisfied, then it is $x\cup L\setminus y$, where $y$ is the smallest upper bound 
of $x$ in $L$; in the latter case, $y$ is said to be {\it bumped}. The new row is denoted $L\leftarrow x$.

The row insertion of $x$ into a tableau $T$, denoted $T\leftarrow x$, is obtained by inserting $x$ into the first row of $T$, and then inserting 
the bumped element into the second row, and so on; the process stops when there is no element bumped.

Columns insertion is defined symmetrically, by replacing rows by columns, and the resulting tableau is denoted $x\to T$.

For later use, we state without proof the following easy lemma.
\begin{lemma}\label{modify} Let $L$ be a row, $x\notin L$, and let $y$ be bumped by $x$ in the row insertion $L\leftarrow x$. 
Define a row $L'$ by modifying the elements of $L$, except $y$, so that no element at the left of $y$ increases; then in the insertion $L'\leftarrow x$, $y$ 
is also bumped.
\end{lemma}

\section{Schensted's lemma}

\begin{theorem} (Schensted \cite{S} Lemma 6) Suppose that $T$ is a tableau, and $x,y$ distinct elements not in $T$. Then
$$(x\to T)\leftarrow y=x\to(T\leftarrow y).$$
\end{theorem}

\section{A trail}

Given a tableau $T$ and a row insertion $T\leftarrow x$, with $x$ not in $T$, we call {\it trail} of 
this insertion the {\it sequence} of 
boxes of $T$,  which are activated in this insertion, with their labels in $T$, followed by the newly created empty box. Formally, if $x$ is 
inserted at the end of the first row
of $T$, then the trail is the newly created box, without label. Otherwise, let $y$ the element bumped from the first row of $T$, and let
$T'$ be the tableau obtained by removing the first row of $T$; then the trail of $T\leftarrow x$ is the box containing $y$ with its label $y$, 
followed by the trail 
of $T'\leftarrow y$. We call {\it empty box of the trail} its last box (which is unlabelled).

The following properties follow easily from the definition of row-insertion:
\begin{itemize}
\item
The consecutive boxes on the trail are on consecutive rows, beginning by the first row of $T$, and they lie on columns that are weakly decreasing; in other words,
the trail goes weakly to the north-west\footnote{We take the French representation of tableaux; for the English one, one has to interchange everywhere north and south, and look at the figures upside down.}; it may go north, but not west.
\item 
If $u,v$ are consecutive labels of the trail, then when $u$ is row-inserted in the row containing $v$, $v$ is bumped, and this defines $v$ uniquely knowing $u$ and the row of $v$;
\item
The labels on a trail are strictly increasing ({\it trail inequality}).

\item
The tableau $T\leftarrow x$ is obtained by sliding each label on the trail to the next box in the trail, and by filling the first box by $x$.
\end{itemize}

For a column insertion $x\to T$, the trail is defined symmetrically, and one has the similar properties:
\begin{itemize}
\item
The consecutive boxes on the trail are on consecutive columns, beginning by the first column of $T$, and they lie on rows that are weakly decreasing; in other words, 
the trail goes weakly to the south-east; it may go east, but not south.
\item 
If $u,v$ are consecutive labels of the trail, then when $u$ is column-inserted in the column containing $v$, $v$ is bumped, and this defines $v$ uniquely knowing $u$ and the column of $v$;
\item
The labels on a trail are strictly increasing ({\it trail inequality}).

\item
The tableau $x\to T$ is obtained by sliding each label on the trail to the next box in the trail, and by filling the first box by $x$.
\end{itemize}

%We sometimes use the expression {\it row trail} to denote the trail of a row-insertion and similarly for {\it column trail}.

\section{Two trails}

We consider now a column insertion $x\to T$ and a row insertion $T\leftarrow y$ of two distinct elements in the same tableau, with $x,y\notin 
T$. 
We therefore have two trails, called here column trail and row trail. We say that they {\it strongly intersect} if they have a box in common.

\begin{lemma}\label{most} The column and row trails have at most one box in common.
\end{lemma}

\begin{proof} Recall that each trail has an empty box, which is its last box. Suppose that the trails intersect in their last box. Since one trail 
goes north-west and the other south-east, this box must be the only intersecting box.

Note that if a box is in the intersection, then it cannot be the empty box for one trail, and nonempty for the other. Indeed, each nonempty box 
of a trail is in $T$, while an empty one is not in $T$.

Suppose now that the two trails have two intersecting boxes, which by what precedes, are not empty. Since the labels strictly increase along 
a trail, and since one goes north-west and the other south-east, we obtain a strictly increasing cycle of numbers; this is impossible.
\end{proof}

In the row trail and in the column trail, we draw a straight segment between the center of any two consecutive boxes; in this way we 
obtain a broken line in the plane, called the {\it geometric trail}; it includes as vertices the center of the boxes of the trail. 

Observe that if $C$ is any center of a box in the tableau, which lies on the geometric trail, then it is a vertex of the latter (since consecutive boxes 
of the row trail lie on consecutive rows of the tableau, and similarly for the column trail). 

We then call  {\it intersection} of the two trails an intersection of the geometric trails. We call it a {\it weak intersection} if this intersection is not strong; that is, in view of what has been just said, if it is not the center of any box in $T$.

\begin{lemma}\label{inter} The trails have no weak intersection.
\end{lemma}

\begin{proof} Suppose by contradiction that the geometric trails have a weak intersection $M$. Then $M$ is not the center of some box in $T$, by the
observation before the statement.

Consider the vertices of the two trails 
closest $M$: $A,B$ in this order for the column trail, and $I,J$ for the row trail. Then $A,B$ are on two consecutive columns and $I,J$ on two consecutive rows.

Suppose by contradiction that $A,B$ are on the same row. Then $M$ is on the open segment $(A,B)$. Thus we have Figure \ref{weak}, right part, where the oblique segment is part of the geometric row trail. On this row trail there must be a box on the same horizontal line as the boxes of $A,B$, and its central point must be on the oblique line; but this is impossible.

Hence $A,B$ are not on the same row, and symmetrically, $I,J$ are not on the same column.
Thus we see that we must have Figure \ref{weak} left part, where $a,b,i,j$ are the labels of $A,B,I,J$. Indeed,  
the boxes must be nonempty: this is clear for the boxes with $i,a$, since they are not the last boxes on their trail; and if, for example, the box 
with $b$ were empty, then it would contradict the tableau property, since we would have this box empty in $T$, but the box with $i$ is in $T$.

Now, since the values increase along a trail, we have $i<j,a<b$.
By the tableau inequalities, we have  $j<a,b<i$. We obtain $a<b<i<j<a$, a contradiction.
%COMPLETER AVEC LES CASES VIDES
\end{proof}

\begin{figure}
\center
\begin{tikzpicture}[scale=0.45]
%\draw (-5,2) -- (-5,3) -- (-4,3) -- (-4,2) -- (-5,2);
%\draw (-5,3) -- (-5,4) -- (-4,4) -- (-4,3) -- (-5,3);
%\draw (-4.5,2.5) -- (-4.5,3.5);
%\draw (-5.5,4.75) -- (-3.5,1.75);

\draw (2,6) -- (2,7) -- (3,7) -- (3,6) -- (2,6);
\draw (0,3) -- (0,4) -- (1,4) -- (1,3) -- (0,3);
\draw (7,2) -- (8,2) -- (8,3) -- (7,3) -- (7,2);
\draw (3,1) -- (4,1) -- (4,2) -- (3,2) -- (3,1);
\draw[dash pattern=on \pgflinewidth off 2pt] (1,3) -- (7,3);
\draw[dash pattern=on \pgflinewidth off 2pt] (3,2) -- (3,6);
\draw (3.4,2) -- (2.6,6);
\draw (1,3.4) -- (7,2.6);
\draw (0.5,3.5) node[scale=0.9]{j};
\draw (7.5,2.5) node[scale=0.9]{i};
\draw (3.5,1.5) node[scale=0.9]{b};
\draw (2.5,6.5) node[scale=0.9]{a};

\draw (12,3) -- (12,4) -- (13,4) -- (13,3) -- (12,3);
\draw (13,3) -- (13,4) -- (14,4) -- (14,3) -- (13,3);
\draw (12.5,3.5) -- (13.5,3.5);
\draw (11.75,4.3) -- (14.75,2.7);
\end{tikzpicture}
\caption{Weak intersection}\label{weak}
\end{figure}
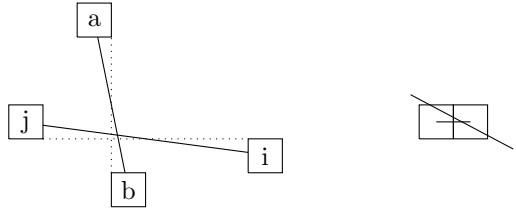

We assume now that the two trails have a strong intersection, which is a box $S$, which we assume nonempty.

\begin{lemma}\label{under} The part of the row trail (resp. column trail) which is before the intersection is strictly under the part of the column-trail (resp. row-trail) which is after it. See Figure \ref{croisement}.
\end{lemma}

\begin{figure}
\includegraphics[scale=1]{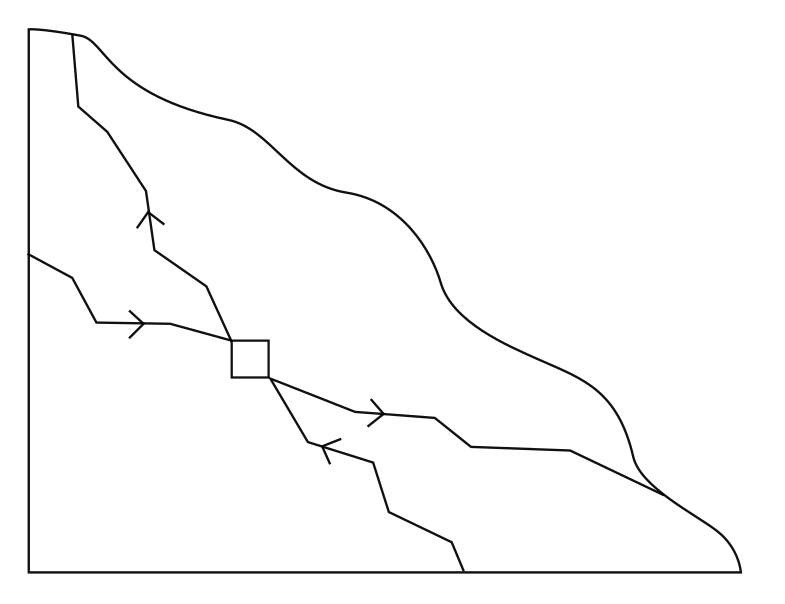}
\caption{Relative position of strongly intersecting row and column trails}\label{croisement}
\end{figure}

\begin{proof} Consider the part of the row trail which is before $S$, and the part of the column trail which is after it. These two parts have no intersection, by Lemmas \ref{most} and \ref{inter}; moreover, the row trail starts from the first row, and the column trail ends at the end of some row. This implies the first statement. The other one is symmetric.
\end{proof}

We denote by $s$ the value in the box $S$ in $T$ and let $\ell$ be the row number of $S$. We denote by $a,s,b$ the consecutive values in the 
trail $x\to T$, and by $i,s,j$ the consecutive values in the trail $T\leftarrow y$. Recall that $a,s,b$ (resp. $i,s,j$) lie in consecutive colomns 
(resp. rows). There is a slight abuse of notation, in the sense that the box after $S$ could be the empty box in either trail, in which case we 
write $b$ or $j=\emptyset$. We also consider the cases where $i$, resp. $a$, does not exist, meaning that $s$ is the first label of the trail of $T\leftarrow x$, resp. $x\to T$; we then put $i=x$, resp. $a=y$.

\begin{lemma}\label{impossible} The configurations in Figure \ref{forbidden} are impossible.
\end{lemma}

Note that the figure means that $a$ is in the column at the left of $s$, and higher than $s$, and so on.

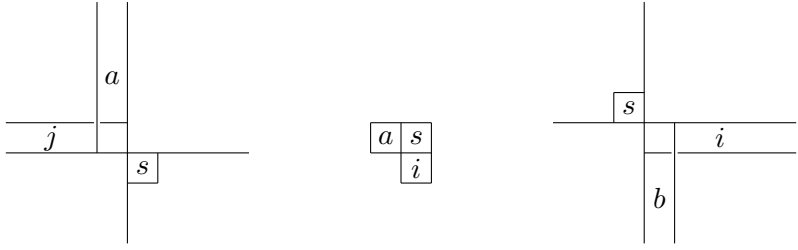
\begin{figure}
\begin{tikzpicture}[scale=0.4]
\draw(0,4) -- (2.9,4); \draw (3.1,4)--(4,4);
\draw(0,3) -- (8,3);
\draw(4,0) -- (4,8);
\draw(3,3) -- (3,8);
\draw(5,3) -- (5,2);
\draw(5,2) -- (4,2);

\draw (1.5,3.5) node{$j$};
\draw (3.5,5.5) node{$a$};
\draw (4.5,2.5) node{$s$};

\draw(12,4) -- (14,4);
\draw(12,3) -- (14,3);
\draw(13,2) -- (14,2);
\draw(12,3) -- (12,4);
\draw(13,2) -- (13,4);
\draw(14,2) -- (14,4);

\draw (13.5,2.5) node{$i$};
\draw (12.5,3.5) node{$a$};
\draw (13.5,3.5) node{$s$};

\draw(18,4) -- (26,4);
\draw(21,3) -- (21.9,3); \draw (22.1,3)--(26,3);
\draw(22,0) -- (22,4);
\draw(21,0) -- (21,8);
\draw(20,5) -- (21,5);
\draw(20,5) -- (20,4);

\draw (21.5,1.5) node{$b$};
\draw (23.5,3.5) node{$i$};
\draw (20.5,4.5) node{$s$};
\end{tikzpicture}
\caption{Impossible configurations (case $b$ or $j=\emptyset$ included)}\label{forbidden}
\end{figure}

\begin{proof} The cases for $j=\emptyset$ at the left, and $b=\emptyset$ at the right, follow from the properties of Ferrers diagrams.

Suppose now that $j,b\neq \emptyset$.
Suppose that we have the leftmost configuration. Then $j\neq  a$ by Lemma \ref{most}; thus $j<a$ by the tableau inequalities; moreover
$a<s$ and $s<j$ by the trail inequalities:
a contradiction.
The rightmost configuration is treated similarly.

For the central one , we have $i>a$, since by row insertion, $i$ bumps $s$. Symmetrically, $a>i$: a contradiction.
\end{proof}

\begin{corollary} One has one of the five configurations of Figure \ref{config}, which indicate the boxes in the trails which have a side in common with $S$.
\end{corollary}

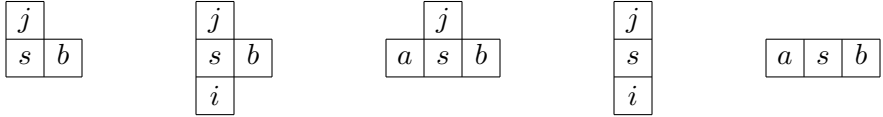
\begin{figure}
\begin{tikzpicture}[scale=0.5]
\draw(0,1) -- (2,1);
\draw(0,2) -- (2,2);
\draw(0,3) -- (1,3);
\draw(0,1) -- (0,3);
\draw(1,1) -- (1,3);
\draw(2,1) -- (2,2);

\draw (0.5,2.5) node{$j$};
\draw (1.5,1.6) node{$b$};
\draw (0.5,1.5) node{$s$};

\draw(5,0) -- (6,0);
\draw(5,1) -- (7,1);
\draw(5,2) -- (7,2);
\draw(5,3) -- (6,3);
\draw(5,0) -- (5,3);
\draw(6,0) -- (6,3);
\draw(7,1) -- (7,2);

\draw (5.5,0.5) node{$i$};
\draw (5.5,2.5) node{$j$};
\draw (6.5,1.6) node{$b$};
\draw (5.5,1.5) node{$s$};

\draw(10,1) -- (13,1);
\draw(10,2) -- (13,2);
\draw(11,3) -- (12,3);
\draw(10,1) -- (10,2);
\draw(11,1) -- (11,3);
\draw(12,1) -- (12,3);
\draw(13,1) -- (13,2);

\draw (11.5,2.5) node{$j$};
\draw (10.5,1.5) node{$a$};
\draw (12.5,1.6) node{$b$};
\draw (11.5,1.5) node{$s$};

\draw(16,0) -- (17,0);
\draw(16,1) -- (17,1);
\draw(16,2) -- (17,2);
\draw(16,3) -- (17,3);
\draw(16,0) -- (16,3);
\draw(17,0) -- (17,3);

\draw (16.5,0.5) node{$i$};
\draw (16.5,2.5) node{$j$};
\draw (16.5,1.5) node{$s$};

\draw(20,1) -- (23,1);
\draw(20,2) -- (23,2);
\draw(20,1) -- (20,2);
\draw(21,1) -- (21,2);
\draw(22,1) -- (22,2);
\draw(23,1) -- (23,2);

\draw (20.5,1.5) node{$a$};
\draw (21.5,1.5) node{$s$};
\draw (22.5,1.6) node{$b$};
\end{tikzpicture}
\caption{Possible configurations}\label{config}
\end{figure}

\begin{proof} Excluding the cases listed in Lemma \ref{impossible} (for example, the left part indicates that $A$ or $J$ must have a side in common with $S$), one verifies that the only possible cases are the five in Figure \ref{config}.
\end{proof}

\section{Three trails}

We consider the trail of the row insertion $T\leftarrow y$ and the trail of the column insertion $x\to T$. We assume that the two trails have a 
strong intersection, which a unique box $S$ labelled $s$. %, that we assume not an empty box in both trails.
The third trail that we study in this section is that of the row insertion $(x\to T)\leftarrow y$. 

The notations $a,b,i,j,s$ are the same as before Lemma \ref{impossible}, and we denote by $A,B,S,I,J$ the boxes that contain these labels.
Say that $S$ is located in row $\ell$.

We compare below the two trails obtained by the two row insertions $T\leftarrow y$ and $(x\to T)\leftarrow y$ (row insertion of $y$ in both 
cases) and begin the task of showing that these two trails are almost equal. When we say below {\it the two trails}, it will always be 
these two row-trails.

\begin{lemma}\label{part} (i) The part of the two trails in rows $1,\ldots,\ell-1$ are equal.

(ii) If the first box in row $\ell+1$ is the same in the two trails, with the same label, then the parts of the two trails above row $\ell$ are equal.
\end{lemma}

Note that it will be shown further that the extra hypothesis in (ii) always holds.

\begin{proof} (i) Consider the rows $1$ to $\ell$ of $T$. When
we apply the insertion $x\to T$, which amounts to slide the elements of $T$ in the trail of $x\to T$, 
we do not modify the part of these rows that are weakly at the left of the trail of $T\leftarrow y$: this follows from Lemma \ref{under}, see Figure \ref{croisement}.
Thus (i) then follows from Lemma \ref{modify}.

(ii) Consider the rows $\ell,\ell+1,\ldots$ of $T$. When we apply the insertion $x\to T$, which amounts to slide the elements of $T$ 
in the trail of $x\to T$, we modify the part of these rows which lie strictly at the left of the trail $T\leftarrow y$ only by decreasing 
elements, or leaving them equal: this follows from Lemma \ref{under}, see Figure \ref{croisement}. Thus (ii) follows from the extra 
hypothesis and Lemma \ref{modify}.
\end{proof}

\begin{proposition}\label{concl} The labels of $S,B,J$ in $(x\to T)\leftarrow y$ are respectively $i,s,a$ if 
$i<a$, and $a,i,s$ if $i>a$. Moreover, the other labels of this tableau are obtained by sliding the labels of the trails of $x\to T$ and $T\leftarrow 
y$, except $i,a,s$, towards the next box of the trail.
\end{proposition}

Note that the slidings are unambiguous.

\begin{figure}

\begin{tikzpicture}[scale=0.5]

\draw(24,4) -- (24,6);
\draw[ultra thick](25,3) -- (25,4) -- (26,4) -- (26,3) -- (25,3);
\draw(25,2) -- (25,6);
\draw(26,3) -- (26,5);
\draw(25,5) -- (26,5);
\draw(24,4) -- (26,4);
\draw(24,3) -- (28,3);
\draw(25,2) -- (28,2);
\draw(24,3) -- (24,4);

\draw (24.5,5) node{$\bf a$};
\draw (27,2.5) node{$i$};
\draw (25.5,4.5) node{$j$};
\draw (25.5,3.5) node{$\bf s$};
\draw (24.5,3.5) node{$t$};

%

%\draw(34,4) -- (34,6);
\draw(35,2) -- (35,5);
\draw(36,3) -- (36,5);
\draw(34,5) -- (36,5);
\draw(34,4) -- (36,4);
\draw(34,3) -- (38,3);
\draw(35,2) -- (38,2);
\draw(34,3) -- (34,5);
\draw[ultra thick](35,3) -- (35,4) -- (36,4) -- (36,3) -- (35,3);

%\draw (34.5,5) node{$a$};
\draw (37,2.5) node{$\bf i$};
\draw (34.5,4.5) node{$v$};
\draw (35.5,4.5) node{$\bf j$};
\draw (35.5,3.5) node{$\bf a$};
\draw (34.5,3.5) node{$t$};

\draw(44,3) -- (44,5);
\draw(43,3) -- (43,5);
\draw(43,3) -- (44,3);
\draw(43,4) -- (44,4);
\draw(43,5) -- (44,5);
\draw[ultra thick](43,3) -- (43,4) -- (44,4) -- (44,3) -- (43,3);

\draw (43.5,4.5) node{$a$};
\draw (43.5,3.5) node{$i$};

\end{tikzpicture}
\caption{Case $i<a$, tableaux $T$, $x\to T$, $(x\to T)\leftarrow y$}\label{i<a}
\end{figure}

\begin{figure}
\begin{tikzpicture}[scale=0.5]

\draw[ultra thick](25,3) -- (25,4) -- (26,4) -- (26,3) -- (25,3);
\draw(24,4) -- (24,6);
\draw(25,2) -- (25,6);
\draw(26,2) -- (26,5);
\draw(23,5) -- (23.9,5);
\draw(24.1,5) -- (24.9,5);
\draw(25.1,5) -- (26,5);
\draw(25.1,4) -- (27,4);
\draw(23,4) -- (23.9,4);
\draw(24.1,4) -- (24.9,4);
\draw(25.1,4) -- (27,4);

\draw(24,3) -- (28,3);
\draw(25,2) -- (28,2);
\draw(24,3) -- (24,4);
\draw(27,3) -- (27,4);

\draw (24.5,5.7) node{$\bf a$};
\draw (27,2.5) node{$i$};
\draw (23.3,4.5) node{$j$};
\draw (25.5,3.5) node{$\bf s$};
\draw (25.5,2.5) node{$t$};
\draw(26.5,3.5) node{$\bf b$};
%
%\draw(34,4) -- (34,6);
\draw(35,3) -- (35,4);
\draw(36,2) -- (36,5);
\draw(33,5) -- (36,5);
\draw(33,4) -- (37,4);
\draw(35,3) -- (38,3);
\draw(36,2) -- (38,2);
%\draw(34,3) -- (34,4);
\draw(37,3) -- (37,4);
\draw[ultra thick](35,3) -- (35,4) -- (36,4) -- (36,3) -- (35,3);

%\draw (34.5,5.5) node{$\bf a$};
\draw (37,2.5) node{$\bf i$};
\draw (33.5,4.5) node{$\bf j$};
\draw (35.5,3.5) node{$a$};
%\draw (35.5,2.5) node{$t$};
\draw(36.5,3.5) node{$\bf s$};
\draw(44,3) -- (44,4);
\draw(43,3) -- (43,4);
\draw(43,3) -- (45,3);
\draw(43,4) -- (45,4);
\draw(45,3) -- (45,4);
\draw[ultra thick](43,3) -- (43,4) -- (44,4) -- (44,3) -- (43,3);

\draw (44.5,3.5) node{$i$};
\draw (43.5,3.5) node{$a$};
\end{tikzpicture}
\caption{Case $i>a$, tableaux $T$, $x\to T$, $(x\to T)\leftarrow y$}\label{i>a}
\end{figure}

\begin{proof} A. We assume that $i,a$ exist. Suppose that $i<a$. In $T$, let $t$ be the label at the left of $s$; since in $T\leftarrow y$, $i$ bumps $s$, we have $i>t$, and 
therefore $t\neq a$; thus $a$ is not at the left of $s$ and it follows by Figure \ref{config} that $j$ is above $s$; moreover $a$ is in the column 
at the left of $s$, and higher than $s$: see Figure \ref{i<a} left part, where the trail of the column insertion of $x$ into $T$ is represented in boldface.

Applying the column insertion of $x$ into $T$, which amounts to slide the trail $\ldots,a,s,b,\ldots$, we obtain $x\to T$ in the central part of Figure \ref{i<a}. Note that $s$ is now in box $B$. Since 
$i<a$, and $i>t$ as we just saw, the trail of $(x\to T)\leftarrow y$, which contains $i$ by Lemma \ref{part} (i), also contains $a$. Now, observe that the element $v$ at 
the left of $j$ in $x\to T$ is $<a$: indeed, either the element $u$ of $T$ in this box is not displaced in the column insertion of $x$ into $T$, so that $u=v$ and $u$ is under $a$ in the same column of $T$; or $u$ is displaced and then $u=a$, and the element $v$ replacing $u$ is 
smaller than $a$. Hence $a<j$ and $a>v$, and the trail of $(x\to T)\leftarrow y$ contains $j$.

Applying the row insertion of $y$ into $x\to T$, we obtain the right part of Figure \ref{i<a}. Note that $s$ is still in box $B$, since the trail of $(x\to T)\leftarrow y$ is strictly under $B$, by Lemmas \ref{under} and \ref{part} (i).
The other assertions follow from Lemma \ref{part}.

We assume now that $i>a$. Let $t$ be the label right under $s$ in $T$; since in the insertion $x\to T$, $a$ bumps $s$, we must have $a>t$, hence $i\neq t$. Thus $i$ is in the row under $s$, stricly at the right of $s$, and by Figure \ref{config}, $b$ is at the right of $s$;
see Figure \ref{i>a} left part; we indicate there that $a$ is in the column at the left of $s$ and in some row weakly above $s$; similarly, $j$ is in the row above $s$ and in some column weakly at the left of $s$. % where the possible positions of $a$ (resp . $j$) are in a row (resp. column) weakly above (resp. weakly at the left of) $s$.

Applying the column insertion $x$ into $T$, which amounts to slide its trail $\ldots a,s,b,\ldots$, we obtain $x\to T$ in the center of Figure \ref{i>a}. Note that $j$ is in box $J$, both in $T$ and in $x\to T$, by Lemma \ref{under}.
The trail of $(x\to T)\leftarrow y$ contains $i$ by Lemma \ref{part} (i), hence also $s$, since $i>a$. Now, if we denote by $v$ the element at the left of $j$, then $v<s$: indeed, in the insertion $T\leftarrow y$, $s$ bumps $j$, hence $s>u$ where $u$ is the element at the left of $j$ in $T$; and $v\leq u$, since the insertion $x\to T$ does not increase elements. Thus $v<s<j$, hence the trail we consider contains $j$. 

The trail is represented in boldface in the center of Figure \ref{i>a}, and we thus obtain $(x\to T)\leftarrow y$ on the right part of Figure \ref{i>a}. The other assertions follow from Lemma \ref{part}.

B. The cases where $a$ or $b$ does not exists is treated quite similarly, with essentially the same arguments. Observe that if, for example, $i$ does not exist, then $s$ must be in the first row and $b$ at its right. We omit the details of these special cases.
\end{proof}

\section{Proof of Schensted's lemma}

Consider the two trails of the insertions $x\to T$ and $T\leftarrow y$. 

If these two trails do not intersect, then the column trail is strictly at the north-west of the row trail, and
there is no interference between 
them: the tableaux $x\to(T\leftarrow y)$ and $(x\to T)\leftarrow y$ are both equal to the tableau obtained from $T$ by sliding the two 
trails of the insertions $x\to T$ and $T\leftarrow y$.

Suppose now that these two trails have an intersection, hence a strong intersection, which is a unique box (Lemmas \ref{inter} and 
\ref{most}). 

Suppose first that this box is the empty box of both trails, denoted $S$. Denoting by $a$ and $i$ the last label in the column and row
trail respectively, it is easy to see that the two tableaux $x\to(T\leftarrow y)$ and $(x\to T)\leftarrow y$ 
are both obtained by sliding the labels in the two trails, except $a$ and $i$ which are added as follows: if $i<a$, then $i$ is put 
in box $S$ and $a$ in the box above; if $i>a$, $a$ is put in box $S$, and $i$ in the box at the right.

Thus we may assume that the intersecting box is unique and not the empty box of either trail. Then we conclude using Proposition \ref{concl} and using the symmetry row/column.

\section{Conclusion}

One may state what happens as follows: one has the two trails of $x\to T$ and $T\leftarrow y$. Then $x\to T \leftarrow y$ is obtained by 
sliding the two trails, except at the boxes where there is a conflict; this is the case for $i$ and $a$, which both want to slide into $S$, 
and for $s$, which has the choice of sliding into $B$ or $J$; the conflict is solved, depending on the inequality between $a$ and $i$, by the 
first sentence in Proposition \ref{concl}. The two cases are represented in Figure \ref{Cases}.

\begin{figure}

\begin{tikzpicture}[scale=0.37]

\draw(25,3) -- (25,5);
\draw(26,3) -- (26,5);
\draw(25,5) -- (26,5);
\draw(25,4) -- (26,4);
\draw(25,3) -- (26,3);
\draw(26,4) -- (27,4);
\draw(27,4) -- (27,1);
\draw(26,3) -- (26,1);
\draw[ultra thick] (25,3) -- (26,3) --(26,4)--(25,4)--(25,3);

\draw (25.5,4.6) node{$j$};
\draw (25.5,3.5) node{$s$};
\draw (26.5,2) node{$b$};
\draw (31,3) node {$\to$};
\draw(35,3) -- (35,5);
\draw(36,3) -- (36,5);
\draw(35,5) -- (36,5);
\draw(35,4) -- (36,4);
\draw(35,3) -- (36,3);
\draw(36,4) -- (37,4);
\draw(37,4) -- (37,1);
\draw(36,3) -- (36,1);
\draw[ultra thick] (35,3) -- (36,3) --(36,4)--(35,4)--(35,3);

\draw (35.5,4.5) node{$a$};
\draw (35.5,3.5) node{$i$};
\draw (36.5,2) node{$s$};
\draw(45,3) -- (45,4);
\draw(46,3) -- (46,5);
\draw(43,5) -- (46,5);
\draw(43,4) -- (46,4);
\draw(45,3) -- (47,3);
\draw(46,4) -- (47,4);
\draw(47,4) -- (47,3);
\draw(46,3) -- (46,3);
\draw[ultra thick] (45,3) -- (46,3) --(46,4)--(45,4)--(45,3);

\draw (46.5,3.5) node{$b$};
\draw (45.5,3.5) node{$s$};
\draw (44,4.5) node{$j$};
\draw (50,3) node {$\to$};

\draw(55,3) -- (55,4);
\draw(56,3) -- (56,5);
\draw(53,5) -- (56,5);
\draw(53,4) -- (56,4);
\draw(55,3) -- (57,3);
\draw(56,4) -- (57,4);
\draw(57,4) -- (57,3);
\draw(56,3) -- (56,3);
\draw[ultra thick] (55,3) -- (56,3) --(56,4)--(55,4)--(55,3);

\draw (56.5,3.5) node{$i$};
\draw (55.5,3.5) node{$a$};
\draw (54,4.5) node{$s$};
\end{tikzpicture}
\caption{the two cases $i<a$ and $i>a$ before and after the row and column insertions}\label{Cases}
\end{figure}
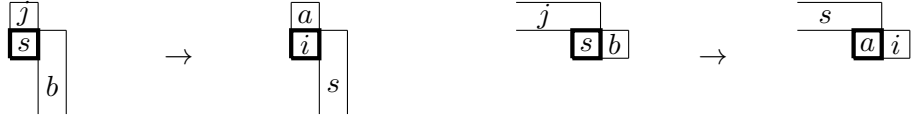

An example is given in the Figure \ref{tab}, where are represented a tableau $T$, the two trails of $7\to T$ and $T\leftarrow 8$, and the final tableau $7\to T\leftarrow 8$; here $i=10, a=11, s=13, j=18, b=14$.

\begin{figure}
\centering
\begin{minipage}{0.33\linewidth}%\centering
\begin{ytableau}
\underline\emptyset \\
17&\underbar{19}\\
\mid\!\!11&  \underbar{18}\\
4&\mid\!\!\underbar{13}&\mid\!\!14 \\
2&6&\underbar {10}&\mid\!\!15&\mid\!\!\emptyset\\
1&3&5&\underbar 9&12&16
\end{ytableau}
%\caption{Trails of the insertions $7\to T$ and $T\leftarrow 8$}
\label{tab1}
\end{minipage}
%\end{figure}
\hspace{1cm}
%\hfill
%\begin{figure}
\begin{minipage}{0.33\linewidth}%\centering
\begin{ytableau}
19 \\
17&18\\
7&11\\
4&10&13 \\
2&6&9&14&15\\
1&3&5&8&12&16
\end{ytableau}
%\caption{$7\to T\leftarrow 8$}\label{tab2}
\end{minipage}
\caption{}\label{tab}
\end{figure}

\medskip
\noindent{\bf Acknowledgments} 
This work was partially supported by NSERC, Canada. Mail exchanges with Bruce Sagan, Richard Stanley, Adriano Garsia, Jean-Yves Thibon, and discussions with Franco Saliola, were helpful. The careful and critical listening of the students of the course ``Algorithmes en combinatoire" at the winter session of 2023 at our university, proved to be very helpful, too.

\end{document}